\documentclass[11pt,a4paper]{amsart}

\usepackage[margin=3.5cm]{geometry}
\usepackage{hyperref}
\usepackage[english]{babel}
\usepackage{amsthm}
\usepackage{amsfonts}
\usepackage{amsmath}
\usepackage{amssymb}
\usepackage{mathrsfs}
\usepackage{enumitem}

\allowdisplaybreaks

\newenvironment{claim}[1]{\par\noindent\textbf{Claim:}\space#1}{}
\newenvironment{claimproof}[1]{\par\noindent\emph{Proof of the Claim:}\space#1}

\newtheorem{theorem}{Theorem}[section]
\newtheorem{lemma}{Lemma}[section]

\newtheorem{corollary}{Corollary}[section]
\theoremstyle{definition}
\newtheorem{remark}{Remark}[section]
\newtheorem{definition}{Definition}[section]
\newtheorem{example}{Example}[section]

\newcommand{\N}{\mathbb{N}}
\newcommand{\Z}{\mathbb{Z}}
\newcommand{\s}{\sigma}
\newcommand{\vt}{\vartheta}
\newcommand{\Pscr}{\mathscr{P}}
\newcommand{\Qscr}{\mathscr{Q}}
\newcommand{\Cscr}{\mathscr{C}}
\newcommand{\dd}{\delta}
\newcommand{\rootd}{\sqrt{\delta}}

\newcommand{\tP}{{\widetilde{P}}}
\newcommand{\sign}[1]{\textup{sign}_P^\eta \langle #1 \rangle_\s}

\newcommand{\ssign}{\textup{sign}}
\newcommand{\sssign}[1]{\textup{sign}_P \langle #1 \rangle_\s}

\newcommand{\qf}[1]{\langle #1\rangle}
\newcommand{\x}{\times}
\newcommand{\ox}{\otimes}
\newcommand{\simto}{\overset\sim\rightarrow}
\DeclareMathOperator{\id}{id}

\DeclareMathOperator{\Sym}{Sym}
\DeclareMathOperator{\Nil}{Nil}

\DeclareMathOperator{\Int}{Int}
\DeclareMathOperator{\Sign}{sign}
\DeclareMathOperator{\Trd}{Trd}

\title[Maximal Prepositive Cones on Quaternion Algebras]{Maximal Prepositive Cones on Quaternion Algebras with Involution}
\author[A. Leader]{Andrew Leader}
\address{St John's College, University of Cambridge, Cambridge CB2~1TP,~UK}
\email{andrewleader@pm.me}
\subjclass{13J30, 16W10, 06F25, 16K20, 11E39}
\keywords{Real algebra, Quaternion algebras with involution, Orderings, Hermitian forms}

\begin{document}

\begin{abstract}
    We give a description of prepositive cones---a notion of ordering on
    algebras with involution introduced by Astier and Unger---in the specific
    context of quaternion algebras with involution. Our main result establishes
    that, for a broad class of quaternion algebras with involution, every
    prepositive cone is maximal.
\end{abstract}

\maketitle

\section{Introduction}

In their investigation of central simple algebras with involution from a real
algebraic point of view, Astier and Unger introduced the notion of prepositive
cones in \cite{PosCones} (see also \cite{A-U-sign}). 
Positive cones are defined as prepositive cones that
are maximal (with respect to inclusion). At the time of writing, examples of
prepositive cones that are not maximal remain elusive, contrary to what is
claimed in \cite[Remark~3.3(2)]{PosCones} (in fact, the authors conveyed in a
private communication that their method did not work and that they intend to 
retract this part of
the remark).
Already in the simplest non-trivial
setting, namely quaternion algebras with involution, it is not clear how to
find such examples. Indeed, in this note we show that the prepositive cones on
a large class of quaternion algebras with involution are all maximal. Up to a
small number of technical results and tools from \cite{Knebusch} and
\cite{PosCones}, our approach is entirely elementary. We also use some results
from elementary calculus which remain valid over arbitrary real closed fields,
cf. \cite[Section~1.3]{S24}.

\section{Preliminaries}\label{Section prelims}

Let $F$ be a formally real field with space of orderings $X_F$. Given an
ordering $P \in X_F$ and $a,b \in F$, we write $a \leq_P b$ if $b-a \in P$ and
let $F_P$ denote a real closure of $F$ at $P$. We denote the unique ordering on
$F_P$ by $\widetilde{P}$. An \emph{$F$-algebra with involution} is a pair
$(A,\s)$, consisting of a finite-dimensional simple $F$-algebra $A$ with centre
a field $K=Z(A)$ and an involution $\s:A\to A$ such that $F=K\cap \Sym(A,\s)$,
where $\Sym(A,\s):=\{x\in A \mid \s(x)=x\}$. Since $F$ is the fixed field of
$\s|_K$, we have $[K:F]\leq 2$. We also recall that $\dim_K(A)=m^2$ for some
$m\in\N$, cf. \cite[1.A]{BOI}.

We say the involution $\s$ is of \textit{the first kind} if $\s|_K$ is the
identity (in which case $K=F$) and of \textit{the second kind} (or
\textit{unitary}) if $\s|_K$ is an automorphism of order~2 (in which case
$[K:F]=2$). If $\s$ is an involution of the first kind, $\s$ can be further
classified as either \textit{orthogonal} if $\dim_K \Sym(A,\s)=m(m+1)/2$ or
\textit{symplectic} if $\dim_K \Sym(A,\s)=m(m-1)/2$. We refer to \cite[2.A,
2.B]{BOI} for further details.

For any subset $S\subseteq A$, we define $S^\x:=S\cap A^\x$, where $A^\x$
denotes the set of invertible elements of $A$. For $u\in \Sym(A,\s)$ we denote
the hermitian form $(x,y)\mapsto \s(x)uy$ by $\qf{u}_\s$. One can show that
$\qf{u}_\s$ is nonsingular if and only if $u\in A^\x$. For $n\in \N$ the
involution $(a_{ij})\mapsto (\s(a_{ji}))$ on $M_n(A)$ is denoted $\s^t$, where
$t$ denotes transposition. For $v \in A^\x$, we denote the inner automorphism
$x \mapsto vxv^{-1}$ by $\Int(v)$.

In \cite{Knebusch} and \cite{A-U-prime}, Astier and Unger studied signatures of
hermitian forms over $F$-algebras with involution. If $h$ is a nonsingular
hermitian form over $(A,\s)$ and $P\in X_F$, they computed the signature of $h$
at $P$ via scalar extension to $F_P$ and hermitian Morita theory. The
computation is relative to the choice of a \emph{reference form} $\eta$, which
can be taken equal to $\qf{1}_\s$ when the involution $\s$ is positive at 
all $Q\in X_F\setminus \Nil[A,\s]$ (see Definition~\ref{def nil orderings} 
below),
i.e., when the involution trace form $(x,y) \mapsto \Trd_A (\s(x)y)$ is
positive definite at all $Q\in X_F\setminus \Nil[A,\s]$, cf. \cite[Definition~4.1,  Remark~4.2 and Corollary~4.6]{A-U-PS}. The
end result is an integer $\Sign_P^\eta h \in \Z$, the \emph{$\eta$-signature of
$h$ at $P$}. Note that replacing $\eta$ by $-\eta$ multiplies the signature map
$\Sign^\eta_\bullet h : X_F \to \Z$ by $-1$, cf. \cite[Propositions~3.2 and
3.3]{A-U-prime}. We have that $\Sign_P^\eta = 0$ whenever $P$ is a nil-ordering
of $(A, \s)$, defined as follows (\cite[Definition~3.7]{Knebusch}):
  
\begin{definition}\label{def nil orderings} 
    We define $\Nil[A,\s]$ to be the set 
    \[
        \begin{cases}
            \{ P \in X_F \ | \ A \ox_F F_P \cong M_{m/2}((-1,-1)_{F_P}) \} & \text{if $\s$ is orthogonal},\\
            \{ P \in X_F \ | \ A \ox_F F_P \cong M_m(F_P) \} & \text{if $\s$ is symplectic},\\
            \{ P \in X_F \ | \ Z(A) \ox_F F_P \cong F_P \x F_P \} & \text{if $\s$ is unitary},
        \end{cases}
    \]
    and call it the set of \emph{nil-orderings} of $(A,\s)$, where the square brackets indicate that $\Nil[A, \s]$ depends only on the Brauer class of $A$ and the type of $\s$.
\end{definition}

\section{Positive Cones}\label{Sect Pos Cones}

We recall \cite[Definition~3.1]{PosCones}:

\begin{definition}\label{def ppc}
    A \textit{prepositive cone on $(A,\s)$} is a subset $\Pscr$ of Sym$(A,\s)$ such that
    \begin{itemize}
        \setlength{\itemindent}{5mm}
        \item[(P1)] $\Pscr \neq \varnothing$;
        \item[(P2)] $\Pscr + \Pscr \subseteq \Pscr$;
        \item[(P3)] $\s(x) \cdot \Pscr \cdot x \subseteq \Pscr$ for every $x \in A$;
        \item[(P4)] $\Pscr_F := \{ u \in F \ | \ u\Pscr \subseteq \Pscr \}$ is an ordering on $F$;
        \item[(P5)] $\Pscr \cap -\Pscr = \{ 0 \}$.
    \end{itemize}
    A prepositive cone $\Pscr$ is \textit{over} $P \in X_F$ if $\Pscr_F = P$. A \textit{positive cone} is a prepositive cone that is maximal with respect to inclusion.
\end{definition}

\begin{remark}\label{remark no pcs over nil}
    By \cite[Proposition~6.6]{PosCones}, positive cones exist only over non-nil orderings. Furthermore, there are exactly two positive cones over any such ordering. If $\Pscr$ is one of them, then $-\Pscr$ is the other, cf. \cite[Theorem~7.5]{PosCones}.
\end{remark}

Given $\varnothing \neq S \subseteq \Sym(A,\s)$ and $P \in X_F$, we denote by $\Cscr_P(S)$ the closure of $S$ under the three operations required by properties (P2), (P3), and (P4) from Definition \ref{def ppc}: 
\[ \Cscr_P (S) := \Bigl\{ \sum_{i=1}^{k} u_i \s(x_i) s_i x_i \ \Big| \ k \in \mathbb{N}, \ u_i \in P, \ x_i \in A, \ s_i \in S \Bigr\}, \]
which is a prepositive cone over $P$ if and only if $\Cscr_P(S)$ satisfies property (P5), cf. \cite[Definition~3.17]{PosCones}.

\begin{example}\label{ex split transpose}
    This is a straightforward modification of \cite[Example~4.3]{A-U-pc-gauge}.
    Let $P\in X_F$ and let $(E,\vt)$ denote either $(F, \id)$, $(F(\rootd),
    \iota)$ with $\dd<_P 0$, or $((a,b)_F, \gamma)$ where $(a,b)_F$ is a
    quaternion division algebra, $\iota$ denotes conjugation, and $\gamma$
    denotes quaternion conjugation. Then $\Sym(E,\vt)=F$, so the only
    (pre)positive cones on $(E,\vt)$ over $P$ are $P$ and $-P$. Furthermore,
    for any $n\in\N$ the only two (pre)positive cones on $(M_n(E),\vt^t)$ are
    the set of positive semidefinite $n\x n$ matrices with respect to $P$ and
    the set of negative semidefinite $n\x n$ matrices with respect to $P$, cf.
    \cite[Proposition~4.10]{PosCones}.

\end{example}

\begin{example}[{\cite[Example~3.13]{PosCones}}]\label{ex div alg}
    Let $(A,\s)$ be a division $F$-algebra with involution and let $P \in X_F \setminus \Nil [A, \s]$. Defining
    \[
        m_P(A,\s) := \max\{\sign{a} \ | \ a \in \Sym(A,\s)^\x\}
    \]
    and
    \[
        \mathscr{M}^\eta_P(A,\s) := \{ a \in \Sym(A,\s)^\x \ | \ \sign{a} = m_P(A,\s) \} \cup \{ 0 \},
    \]
    we have that $\mathscr{M}^\eta_P(A,\s)$ and $-\mathscr{M}^\eta_P(A,\s)$ are the only positive cones on $(A,\s)$ over $P$ by \cite[Proposition~7.1]{PosCones}. 
\end{example}

\section{Quaternion Algebras with Involution} \label{Sect Quat alg}

In this paper we are only interested in $F$-algebras with involution for which $\dim_K A=4$, the \emph{quaternion $F$-algebras with involution}. 
Recall that a quaternion algebra over $K$ is a central simple $K$-algebra
with $K$-basis $\{1,i,j,k\}$ subject to the relations 
\[
    i^2 = a, \quad j^2 = b, \quad ij = k = -ji, 
\]
for some $a,b \in K^\x$, and is denoted by $(a,b)_K$. Furthermore, $A:=(a,b)_K$
is either a division algebra or is \emph{split}, i.e., $A \cong M_2(K)$. The
canonical involution $\gamma$ on $A$ (called quaternion conjugation), defined
by $\gamma(a_0 + a_1i + a_2j +a_3k) = a_0 - a_1i - a_2j -a_3k$ for
$a_0,a_1,a_2,a_3 \in K$, is the unique symplectic involution on $A$. By the
Skolem-Noether theorem, any orthogonal involution $\s$ on $A$ is of the form
$\s=\Int(v)\circ \gamma$, for some invertible element $v \in A\setminus K$ such
that $\gamma(v)=-v$. The element $v$ is uniquely determined by $\s$ up to a
factor in $K^\x$, cf. \cite[Proposition~2.21]{BOI}. By a theorem of Albert, cf.
\cite[Proposition~2.22]{BOI}, if $\s$ is a unitary involution on $A$, there
exists a unique quaternion $F$-subalgebra $A_0\subset A$ such that $A=A_0\ox_F
K$ and $\s=\gamma_0\ox \iota$, where $\gamma_0$ is the canonical involution on
$A_0$ and $\iota=\s|_K$.

For the problem considered in this paper, we only need to consider nonsingular
hermitian forms $\qf{u}_\s$ over quaternion algebras with involution. We recall
how to compute $\eta$-signatures in this particular case, and refer to
\cite{Knebusch} and \cite{A-U-prime} for the general case.
  
Let $(A,\s)$ be a quaternion $F$-algebra with involution. Let $\eta$ be a
reference form for $(A,\s)$ and let $h=\qf{u}_\s$ with $u\in \Sym(A,\s)^\x$.
Let $P\in X_F$. If $A\ox_F F_P\cong M_2(F_P)$ and $\s$ is symplectic, or if
$A\ox_F F_P\cong (-1,-1)_{F_P}$ and $\s$ is orthogonal, or if $K\ox_F F_P \cong
F_P\x F_P$ and $\s$ is unitary, then $P$ is a nil-ordering by Definition
\ref{def nil orderings} and we set $\Sign^\eta_P h:=0$.
  
In the remaining cases (i.e., when $P$ is not a nil-ordering) we proceed as 
follows: 

\begin{enumerate}[label={\rm{(\alph*)}}]

\item If $\s$ is orthogonal, we consider a splitting isomorphism
    \[
        \lambda_P: A\ox_F F_P\simto M_2(F_P).
    \]
    By the Skolem-Noether theorem the involution $\s\ox\id$ corresponds to the
    involution $\Int(\Phi_P)\circ t$ under $\lambda_P$, for some invertible
    symmetric matrix $\Phi_P \in M_2(F_P)$. Then $\Phi_P^{-1} \cdot
    \lambda_P(u\ox 1)$ is the matrix of a nonsingular symmetric bilinear form
    $b_h$ over $F_P$ with Sylvester signature $\Sign_\tP b_h$. We compute the
    signature at $\tP$ of the reference form $\eta$ in a similar manner. (The
    details are not important for our purposes, but we do note that this
    signature is non-zero, cf. \cite[Propositions~3.2 and 3.3]{A-U-prime}.) Let
    $s_P\in\{-1,+1\}$ denote the sign of this signature. Then we define
    $\Sign_P^\eta h:= s_P\cdot \Sign_\tP b_h$.\smallskip

\item If $\s$ is symplectic, and $A\ox_F F_P\cong (-1,-1)_{F_P}$, then $A$ is
    a division algebra and $\s=\gamma$, quaternion conjugation. Then
    $u\in\Sym(A,\gamma)^\x=F^\x$ and we define $\Sign_P^\eta h$ to be the sign
    of $u$ at $P$. (We note that we can take $\eta=\qf{1}_\gamma$ since 
    $\gamma$ is
    positive at $P$.)\smallskip

\item If $\s$ is unitary, then $K\ox_F F_P\cong F_P(\sqrt{-1})$, and we
    compute $\Sign_P^\eta h$ as in the split-orthogonal case (a) above with the
    following obvious changes: we replace $(F_P, \id)$ by $(F_P(\sqrt{-1}),
    \iota_P)$, where $\iota_P$ denotes conjugation; $\Phi_P \in
    M_2(F_P(\sqrt{-1})) $ is an invertible hermitian matrix this time and $b_h$
    is a nonsingular hermitian form over $(F_P(\sqrt{-1}), \iota_P )$.
\end{enumerate}
\section{Main Result}

At the time of writing there are no known examples of prepositive cones that
are not maximal. The main result of this paper shows that such examples cannot
be found for a large class of quaternion algebras with involution. We will use
the order topology on $(F,P)$ and $(F_P,\tP)$, cf. \cite[Section~2.6]{KSU}.

\begin{theorem}\label{theorem main}
    Let $(A,\s)$ be a quaternion $F$-algebra with involution and let $P \in X_F$. If $F$ is dense in $F_P$, then all prepositive cones on $(A,\s)$ over $P$ are maximal.
\end{theorem}

\begin{remark} 
  In the process of proving Theorem~\ref{theorem main} we obtain the following
  observations given $P\in X_F\setminus \Nil[A,\s]$ (noting that the theorem is 
  vacuously true if $P\in \Nil[A,\s]$ by Remark~\ref{remark no pcs over nil}), 
  the first two of which do not require $F$ to be dense in $F_P$:
  
  \begin{enumerate}[label={\rm{(\arabic*)}}]
  \item All prepositive cones  on quaternion $F$-algebras with involution 
  over $P$ are maximal if and only if this is the case 
  for all $(A,\s)$ with $A$ a quaternion division algebra and such that, 
  either 
    \begin{enumerate}[label={\rm{(\roman*)}}]
      \item $A=(a,b)_F$ where $a,b>_P0$ and
      $\s$ is the orthogonal involution that fixes $i$ and $j$, or
      \item $A=(a,b)_F \ox_F F(\sqrt{\delta})$ where $a,b>_P0$, $\delta<_P 0$ and
      $\s$ is the unitary involution that restricts to the canonical 
      symplectic involution on $(a,b)_F$.
    \end{enumerate}

    \item In case {\rm (i)}, all prepositive cones
    are maximal if and only if $\Cscr_P(1)$ is maximal. In case {\rm (ii)}, all 
    prepositive cones are maximal if and only if $\Cscr_P(k\ox \sqrt{\delta})$ 
    is maximal.
    
    \item If $F$ is dense in $F_P$, then $\Cscr_P(1)$ 
    is maximal in case {\rm (i)}, 
    and $\Cscr_P(k\ox \sqrt{\delta})$  is maximal in 
    case {\rm (ii)}.    
  \end{enumerate}
  We also note that by hermitian Morita theory, the result naturally extends
  to any $F$-algebra of index $2$ with involution $(A,\s)$. Indeed, 
  this follows from  \cite[Theorem~4.2]{PosCones} since
  any quaternion 
  division algebra that is Brauer equivalent to $A$ will carry an involution
  of the same type as $\s$ (this is clear if $\s$ is of the first kind and
  follows from \cite[Theorem~3.1]{BOI} otherwise).
\end{remark}

    Before proving the theorem, we mention a consequence and an auxiliary result.

\begin{corollary}
    Let $(A,\s)$ be a quaternion $F$-algebra with involution and let $P$ be an archimedean ordering on $F$. Then all prepositive cones on $(A,\s)$ over $P$ are maximal.  
\end{corollary}
\begin{proof}
    If $P$ is archimedean, then it follows from \cite[Theorem~1.1.5]{PD} that $F$ is dense in $F_P$.
\end{proof}

\begin{lemma}\label{lemma minimum}
    Let $P \in X_F$ and let $u,v \in F$ be positive at $P$. Consider 
    \[
        f : F_P \setminus \{0\} \to F_P, \ x \mapsto ux^2 + \frac{v}{x^2}.
    \]
    \begin{enumerate}[label={\rm{(\arabic*)}}]
        \item The function $f$ attains its minimum $2\sqrt{uv}$ at $x^* = (\frac{v}{u})^\frac{1}{4}$.
        \item Assume that $F$ is dense in $F_P$. For any $t \in F_P$ with $t >_\tP f(x^*)$, there exists $\beta \in F$ such that $f(\beta) \in (f(x^*),t) \subset F_P \setminus \{0\}$.
    \end{enumerate}
\end{lemma}
\begin{proof}
    (1): This follows from some results from elementary calculus which remain valid over any real closed field, cf. \cite[Section~1.3]{S24}.

    (2): Consider the open interval $I:=(f(x^*),t)\subset F_P\setminus\{0\}$. Note that $f$ is a non-constant continuous function on $F_P\setminus\{0\}$ that attains its minimum at $x^*$ by (1). As a result, $J:=f^{-1}(I)$ is a non-empty open
    subset of $F_P\setminus\{0\}$, cf. \cite[Corollary~1.3.10(b)]{S24}. Since $F$ is assumed dense in $F_P$, there
    exists a nonzero element $\beta \in J\cap F$. The statement follows.
\end{proof}

\begin{proof}[Proof of Theorem \ref{theorem main}]  
  Since (pre)positive cones are preserved under hermitian Morita 
  equivalence by \cite[Theorem~4.2]{PosCones}, 
  we can freely replace $(A,\s)$ by any algebra in the Brauer class of $A$ and 
  any 
  involution of the same type as $\s$.  
  
    Assume first that $A$ is split, i.e., $A \cong M_2(K)$. If $\s$ is
    symplectic, then $P \in \Nil[A,\s]$ since $K=F$ and clearly $A \ox_F F_P
    \cong M_2(F_P)$. Thus, by Remark \ref{remark no pcs over nil}, there are no
    (pre)positive cones on $(A,\s)$. 
    
    If $\s$ is orthogonal or unitary, then $\s$ is of the same
    type as the conjugate transpose involution $\iota^t : (a_{k\ell})
    \mapsto (\iota(a_{\ell k}))$, where $\iota=\id_F$ in case $\s$ is 
    orthogonal. Thus we may assume that $\s = \iota^t$.
    Hence, there are only two prepositive cones on $(A,\s)$ and they are both
    maximal by Example~\ref{ex split transpose}.

    We may therefore assume that $A$ is a quaternion division algebra. It follows from \cite[Proposition~4.4]{A-U-PS} that $m_P(A,\s) \leq 2$. In fact, in all relevant cases below, we will see that $m_P(A,\s)=2$. Hence, by Example \ref{ex div alg},
    \[
    \Pscr := \{ a \in \Sym(A,\s)^\x \ | \ \sign{a}=2 \} \cup \{ 0 \}
    \] 
    and $-\Pscr$ are the only positive cones on $(A, \s)$ over $P$,
    where $\eta$ is a fixed reference form for $(A,\s)$.
    Note that if we replace $\eta$ by $-\eta$, then $\Pscr$ is replaced by
    $-\Pscr$. Thus, without loss of generality, we may assume that $s_P=1$
    and we simply write $\ssign_P$ instead of $\ssign_P^\eta$. 
    \medskip
    
    \noindent\emph{Case 1: ${\s}$ symplectic.} In this case, $\s$ is quaternion conjugation. Thus, $\Sym(A,\s) = F$ and it follows that all
prepositive cones on $(A,\s)$ over $P$ are maximal by Example \ref{ex split
transpose}.
\medskip
    
    \noindent\emph{Case 2: ${\s}$ orthogonal.} We write $A = 
    (a,b)_F$ with $a,b \in F^\x$. If $a<_P 0, \, b<_P 0$, we have that 
    $A \ox_F F_P
\cong (-1,-1)_{F_P}$. As a result, $P \in \text{Nil}[A,\s]$ since $\s$ is
orthogonal and there are no (pre)positive cones on $(A,\s)$ over $P$ by
Remark~\ref{remark no pcs over nil}. 

For the remaining cases we may assume without loss
of generality that $a>_P 0$ and $b>_P 0$ using
the isomorphisms 
$(a,b)_F\cong (b,a)_F$ and $(a,b)_F\cong (a,-ab)_F$. Furthermore, 
by the observation at the start of the proof,
we may assume without loss of generality
that $\s$ is the involution that sends $k$ to $-k$ and fixes $1,i,$ and $j$.
Consider the splitting isomorphism $\lambda_P: A \ox_F F_P \rightarrow
M_2(F_P)$ given by the assignments:
    \begin{align*}
        1 \ox 1 &\mapsto \begin{bmatrix}
            1 & 0\\
            0 & 1
        \end{bmatrix},&
        i \ox 1 &\mapsto \begin{bmatrix}
            \sqrt{a} & 0\\
            0 & -\sqrt{a}
        \end{bmatrix},\\
        j \ox 1 &\mapsto \begin{bmatrix}
            0 & \sqrt{b}\\
            \sqrt{b} & 0
        \end{bmatrix},&
        k \ox 1 &\mapsto \begin{bmatrix}
            0 & \sqrt{ab}\\
            -\sqrt{ab} & 0
        \end{bmatrix}.
    \end{align*}
    Under $\lambda_P$, the involution $\s \ox \id$ corresponds to the transposition
involution $X \mapsto X^t$. We see that $\ssign_P\qf{1}_\s=2$, and so 
$m_P(A,\s)=2$.
    \medskip
 
    \begin{claim}
         Let $d = d_0 + d_1i + d_2j \in \Sym(A,\s)$. Then 
         $\sssign{d} = \pm 2$  if and only if $d_0^2 >_P ad_1^2 + 
         bd_2^2$, in which case $\sssign{d} = 2$ if and only if $d_0 >_P 0$.
    \end{claim}

    \begin{claimproof}
        Since $\Phi_P = \begin{bmatrix}
             1 & 0\\
             0 & 1
        \end{bmatrix} = \Phi_P^{-1}
        $,
        we have that
        \[
            \Phi_P^{-1} \cdot \lambda_P(d \ox 1) = \begin{bmatrix}
            d_0+d_1\sqrt{a} & d_2\sqrt{b} \\
            d_2\sqrt{b} & d_0-d_1\sqrt{a}
            \end{bmatrix}.
        \]
        By the principal axis theorem (which holds for the real closed field $F_P$), this matrix is congruent to the diagonal matrix
        \[
            \begin{bmatrix}
                d_0-\sqrt{ad_1^2+bd_2^2} & 0 \\
                0 & d_0+\sqrt{ad_1^2+bd_2^2}
            \end{bmatrix}.
        \]
        The claim follows by considering the positivity of elements along the diagonal, or equivalently, the determinant.\qed
    \end{claimproof}

    \medskip

    Assume now that $\Qscr$ is a prepositive cone on $(A,\s)$ over $P$ and that
    $\Qscr \subseteq \Pscr$. We show that $\Cscr_P(1) \subseteq \Qscr$ in this
    case. Indeed, let $u = u_0 + u_1i + u_2j \in \Qscr \setminus \{0\}$. Then
    $\Cscr_P(u) \subseteq \Qscr$ and $\Cscr_P(u)$ is a prepositive cone (since
    $\Qscr$ is one). Notice that
    \[
        \begin{aligned}
            \s(i)ui &= a(u_0 + u_1i - u_2j), \\
            \s(j)uj &= b(u_0 - u_1i + u_2j). 
        \end{aligned}
    \]
    Thus, $1=\frac{1}{2 a u_0} \s(i)ui + \frac{1}{2 b u_0} \s(j)uj \in
    \mathscr{C}_P(u)$ since $u_0 >_P 0$ by the above claim. Therefore,
    $\mathscr{C}_P(1) \subseteq \mathscr{C}_P(u) \subseteq \mathscr{Q}
    \subseteq \mathscr{P}$. Likewise, if we assume $\Qscr \subseteq -\Pscr$ and
    let $u \in \Qscr\setminus \{0\}$, then $\mathscr{C}_P(-1) \subseteq
    \mathscr{C}_P(u) \subseteq \mathscr{Q} \subseteq -\mathscr{P}$.

    To finish the proof in this case, we will show that $\Cscr_P(1)= \Pscr$.
    Therefore, also $\Cscr_P(-1) = -\Pscr$. This then implies that $\Qscr =
    \Pscr$ or $\Qscr = -\Pscr$, i.e., that $\Qscr$ is always maximal. As we
    immediately have that $\Cscr_P(1) \subseteq \Pscr$, we need only show the
    opposite inclusion. Let $d = d_0 + d_1i +d_2j \in \Pscr \setminus \{0\}$.
    We wish to show that $d \in \Cscr_P(1)$, i.e., we must attain an expression
    $d = \sum_{n=1}^{N} u_n \s(x_n) x_n$ for some $N \in \N$, $u_n \in P$, and
    $x_n \in A$.

    Consider the function 
    \[
        f: F_P \setminus \{0 \} \to F_P,\ x \mapsto x^2+\frac{ad_1^2 + 
        bd_2^2}{4x^2}. 
    \]
    Since $d\in \Pscr \setminus \{0\}$, we know by the claim that $d_0^2 >_P
    ad_1^2 +bd_2^2$ and $d_0 >_P 0$. Considered as elements of $F_P$, we can
    write $d_0 >_{\tP} \sqrt{ad_1^2 +bd_2^2} $. By Lemma \ref{lemma minimum},
    the continuous function $f$ attains its minimum $\sqrt{ad_1^2 + bd_2^2}$ at
    $x^* = \frac{1}{\sqrt{2}}(ad_1^2+bd_2^2)^\frac{1}{4}$ and there exists
    $\beta \in F$ such that $f(\beta) \in (\sqrt{ad_1^2 +bd_2^2}, d_0)\subset
    F_P\setminus\{0\}$. Then $f(\beta)<_{\tP} d_0$, and since $f(\beta)$ and
    $d_0$ are both in $F$, actually $f(\beta)<_P d_0$. It follows that
    $d_0-f(\beta)>_P 0$, and so $d_0-f(\beta)=(d_0-f(\beta))\s(1)1 \in
    \Cscr_P(1) $.

    Let 
    $c=\beta +\tfrac{d_1}{2\beta}i +\tfrac{d_2}{2\beta}j$.  Then
    $\s(c)c=f(\beta) +d_1i +d_2j \in \Cscr_P(1)$. Finally,
    $d=\s(c)c + (d_0-f(\beta))\s(1)1 \in \Cscr_P(1)$, as required. \medskip

    \noindent\emph{Case 3: $\s$ unitary.} As remarked in 
    Section~\ref{Sect Quat alg}, we can write $A = A_0 \ox_F K$, where $K=
F(\rootd)$ for some $\dd \in F^\x$ such that $[K:F] =2$, $A_0$ is a quaternion
algebra over $F$, and $\s = \gamma_0 \ox \iota$. Note that $A_0$ is a division
algebra since $A$ is assumed to be a division algebra. If $ \dd >_P 0$, then $K
\ox_F F_P \cong F_P \x F_P$, and so $P$ is a nil-ordering, which precludes the
existence of (pre)positive cones by Remark~\ref{remark no pcs over nil}. Hence,
we may further assume $\dd <_P 0$. We write $A_0=(a,b)_F$ for some $a,b \in
F^\x$. 
Without loss
of generality we may choose $a>_P 0$ and $b>_P 0$ using
the isomorphisms 
$(a,b)_F\cong (b,a)_F$, $(a,b)_F\cong (a,-ab)_F$, and
$(a,b)_F\ox_F K \cong (a\delta, b\delta)_F \ox_F K$ (since $\delta$ is a square 
in $K=F(\sqrt{\delta})$).

Let $\{1,i,j,k\}$ be the usual $F$-basis of $A_0$. Since $K$ has an
$F$-basis $\{1, \rootd\}$, we then have an $F$-basis for $A$ given by $\{1\ox1,
i\ox1, j\ox1, k\ox1, 1\ox\rootd, i\ox\rootd, j\ox\rootd, k\ox\rootd \}$. An
$F$-basis for $\Sym(A,\s)$ is given by $\{1\ox 1, i \ox \rootd, j \ox \rootd, k
\ox \rootd\}$.

Consider the splitting map $\lambda_P : A \ox_F F_P \to M_2(F_P(\sqrt{-1}))$ given by the assignments:
    \begin{align*}
        (1 \ox 1) \ox 1 &\mapsto 
        \begin{bmatrix}
            1 & 0 \\ 0 & 1
        \end{bmatrix},
        &
        (1 \ox \rootd) \ox 1 &\mapsto 
        \begin{bmatrix} 
            \rootd & 0 \\ 0 & \rootd
        \end{bmatrix}, \\
        (i \ox 1) \ox 1 &\mapsto 
        \begin{bmatrix}
            \sqrt{a} & 0 \\ 0 & -\sqrt{a}
        \end{bmatrix},
        &
        (i \ox \rootd) \ox 1 &\mapsto 
        \begin{bmatrix} 
            \sqrt{a\dd} & 0 \\ 0 & -\sqrt{a\dd}
        \end{bmatrix}, \\
        (j \ox 1) \ox 1 &\mapsto 
        \begin{bmatrix}
            0 & \sqrt{-b} \\ -\sqrt{-b} & 0
        \end{bmatrix},
        &
        (j \ox \rootd) \ox 1 &\mapsto 
        \begin{bmatrix} 
            0 & \sqrt{-b\dd} \\ -\sqrt{-b\dd} & 0
        \end{bmatrix},\\
        (k \ox 1) \ox 1 &\mapsto 
        \begin{bmatrix} 
            0 & \sqrt{-ab} \\ \sqrt{-ab} & 0
        \end{bmatrix},
        &
        (k \ox \rootd) \ox 1 &\mapsto 
        \begin{bmatrix} 
            0 & \sqrt{-ab\dd} \\ \sqrt{-ab\dd} & 0
        \end{bmatrix}.
    \end{align*}
    Under $\lambda_P$, the involution $\s \ox \id = (\gamma_0 \ox \iota) \ox \id$ corresponds to the involution
    \[
        X \mapsto \Phi_P \, \overline{X}^t \Phi_P^{-1}, \quad \text{where} \quad \Phi_P = \begin{bmatrix}
            0 & 1\\
            1 & 0
        \end{bmatrix} = \Phi_P^{-1},
    \]
    and where $-^t$ is the conjugate transpose involution induced by the map $\sqrt{-1} \mapsto -\sqrt{-1}$. Since
    \[
        \Phi_P^{-1} \cdot \lambda_P((k \ox \rootd) \ox 1) = 
            \begin{bmatrix}
                \sqrt{-ab\dd} & 0\\
                0 & \sqrt{-ab\dd}
            \end{bmatrix},
    \]
    we have $\sssign{k \ox \rootd} = 2$. 
    In particular, $m_P(A,\s)=2$.
    \medskip
    \begin{claim}
        Let $d = d_0(1\ox1) + d_5(i\ox\rootd) + d_6(j\ox\rootd) +d_7(k\ox\rootd) \in \Sym(A,\s)$. Then $\sssign{d} = \pm 2$ if and only if $d_7^2 >_P \frac{1}{-ab\dd} (d_0^2 - a\dd d_5^2 - b\dd d_6^2)$, in which case, $\sssign{d} = 2$ if and only if $d_7 >_P 0$.
    \end{claim}

    \medskip
    
    \begin{claimproof}
        We have that
        \[
            \Phi_P^{-1} \cdot \lambda_P(d \ox 1) = 
            \begin{bmatrix}
                -d_6\sqrt{-b\dd} +d_7\sqrt{-ab\dd} & d_0-d_5\sqrt{a\dd}\\
                d_0+d_5\sqrt{a\dd} & d_6\sqrt{-b\dd}+d_7\sqrt{-ab\dd}
            \end{bmatrix}.
        \]
        By the principal axis theorem (the hermitian version of which holds for 
        the field $F_P(\sqrt{-1})$), this 
        matrix is congruent to the diagonal matrix
        \[
            \begin{bmatrix}
                d_7\sqrt{-ab\dd}-\sqrt{d_0^2-a\dd d_5^2-b\dd d_6^2} & 0\\
                0 & d_7\sqrt{-ab\dd}+\sqrt{d_0^2-a\dd d_5^2-b\dd d_6^2}
            \end{bmatrix}.
        \]
        The claim follows by considering the positivity of elements along the diagonal, or equivalently, the determinant.\qed
    \end{claimproof}

    \medskip

    Assume now that $\Qscr$ is a prepositive cone on $(A, \s)$ over $P$ and
    that $\Qscr \subseteq \Pscr$. We show that $\Cscr_P(k \ox \rootd) \subseteq
    \Qscr$. Indeed, let $u = u_0(1\ox 1) + u_5(i\ox \rootd) + u_6(j\ox \rootd)
    + u_7(k\ox \rootd) \in \Qscr\setminus \{0\}$. Then $\Cscr_P(u) \subseteq
    \Qscr$ and
    $\Cscr_P(u)$ is a prepositive cone (since $\Qscr$ is one). Notice that
    \begin{align*}
        x_1 := &\ \s(1\ox 1) u (1\ox 1)\\
        = &\ u_0 (1\ox 1) + u_5(i \ox \rootd) + u_6 (j\ox \rootd) + u_7(k \ox \rootd),\\
        x_2 := &\ \s(i\ox \rootd) u (i\ox \rootd)\\
        = &\ a\dd (u_0 (1\ox 1) + u_5(i \ox \rootd) - u_6 (j\ox \rootd) - u_7(k \ox \rootd)),\\
        x_3 := &\ \s(j\ox \rootd) u (j\ox \rootd)\\
        = &\ b\dd (u_0 (1\ox 1) - u_5(i \ox \rootd) + u_6 (j\ox \rootd) - u_7(k \ox \rootd)),\\
        x_4 := &\ \s(k\ox \rootd) u (k\ox \rootd)\\
        = &\ ab\dd (-u_0 (1\ox 1) + u_5(i \ox \rootd) + u_6 (j\ox \rootd) - u_7(k \ox \rootd)),
    \end{align*}
    allowing us to express $k\ox \rootd$ as a typical element of $\Cscr_P(u)$: We have
    \[
        k\ox \rootd = \frac{x_1}{4u_7} -\frac{x_2}{4a\dd u_7} -\frac{x_3}{4b\dd u_7} -\frac{x_4}{4ab\dd u_7} \in \Cscr_P(u)
    \]
    since $u_7 >_P 0$ by the above claim and $\dd <_P 0$. Therefore, $\Cscr_P(k\ox \rootd) \subseteq \Cscr_P(u) \subseteq \Qscr \subseteq \Pscr$. Likewise, if $\Qscr \subseteq -\Pscr$ and $u \in \Qscr\setminus \{0\}$ then $\Cscr_P(-k\ox \rootd) \subseteq \Cscr_P(u) \subseteq \Qscr \subseteq -\Pscr$.

    To finish the proof in this case, we will show that $\Cscr_P(k \ox \rootd) = \Pscr$. Therefore, also $\Cscr_P(-k \ox \rootd) = -\Pscr$. This then implies that $\Qscr = \Pscr$ or $\Qscr = -\Pscr$, i.e., that $\Qscr$ is always maximal. As we immediately have that $\Cscr_P(k \ox \rootd) \subseteq \Pscr$, we need only show the opposite inclusion. Let $d = d_0(1\ox1) + d_5(i\ox\rootd) + d_6(j\ox\rootd) +d_7(k\ox\rootd) \in \Pscr \setminus \{0\}$. To prove $\Pscr \subseteq \Cscr_P(k \ox \rootd)$, we wish to show that $d = \sum_{n=1}^{N} u_n \s(x_n) (k \ox \rootd) x_n$ for some $N \in \N$, $u_n \in P$, and $x_n \in A$, so that $d \in \Cscr_P(k \ox \rootd)$.

    Consider the function
    \[
        f: F_P \setminus \{0\} \to F_P,\ x \mapsto - ab\dd x^2 + \frac{d_0^2-a\dd d_5^2 -b\dd d_6^2 }{4a^2b^2\dd^2x^2}.
    \]
    Since $d \in \Pscr \setminus \{0\}$, we know by the claim that $d_7^2 >_P \frac{1}{-ab\dd} (d_0^2 - a\dd d_5^2 - b\dd d_6^2)$ and that $d_7 >_P 0$. Considered as elements of $F_P$, we can write $d_7 >_\tP \frac{1}{\sqrt{-ab\dd}}\sqrt{d_0^2-a\dd d_5^2-b\dd d_6^2}$. By Lemma \ref{lemma minimum}, the continuous function $f$ attains its minimum 
    \[
        \frac{1}{\sqrt{-ab\dd}}\sqrt{d_0^2-a\dd d_5^2-b\dd d_6^2} \quad \text{at}\quad x^* = \frac{(d_0^2-a\dd d_5^2-b\dd d_6^2)^{\frac{1}{4}}}{(-4a^3b^3\dd^3)^{\frac{1}{4}}},
    \]
    and there exists $\beta \in F$ such that 
    \[
    f(\beta) \in \left( \frac{1}{\sqrt{-ab\dd}}\sqrt{d_0^2-a\dd d_5^2-b\dd d_6^2}, d_7 \right) \subset F_P \setminus \{0\}.
    \]
    Then $f(\beta)<_{\tP} d_7$, and since $f(\beta)$ and $d_7$ are
    both in $F$, actually $f(\beta)<_P d_7$. It follows that 
    $d_7-f(\beta)>_P 0$, and so $(d_7-f(\beta))(k \ox \rootd)=(d_7-f(\beta))\s(1 \ox 1)(k \ox \rootd)(1 \ox 1) \in \Cscr_P(k \ox \rootd)$.

    Let $c= \frac{-d_0}{2ab\dd \beta}(1 \ox 1) - \frac{d_5}{2ab\dd \beta}(i \ox \rootd) - \frac{d_6}{2ab \dd \beta}(j \ox \rootd) + \beta (k \ox \rootd)$. Then 
    \[
        \s(c)(k \ox \rootd)c = d_0(1 \ox 1) + d_5(i \ox \rootd) + d_6(j \ox \rootd) 
        + f(\beta)(k \ox \rootd) \in \Cscr_P (k \ox \rootd).
    \]
    Finally, we can write 
    \[
        d = \s(c)(k \ox \rootd)c + (d_7-f(\beta))\s(1 \ox 1)(k \ox \rootd)(1 \ox 1) \in \Cscr_P(k \ox \rootd),
    \]
    as required.
\end{proof}

\section*{Acknowledgments}
I would like to thank Thomas Unger for our discussions on the topic of this paper.



\begin{thebibliography}{00}

\bibitem{Knebusch}
V. Astier and T. Unger,
\newblock Signatures of hermitian forms and the {K}nebusch trace formula,
\newblock {\it Math. Ann.} {\bf 358}(3-4) (2014) 925--947.

\bibitem{A-U-prime}
V. Astier and T. Unger,
\newblock Signatures of hermitian forms and ``prime ideals'' of {W}itt groups,
\newblock {\it Adv. Math.} {\bf 285} (2015) 497--514.

\bibitem{A-U-PS}
V. Astier and T. Unger,
\newblock Signatures of hermitian forms, positivity, and an answer to a question of {P}rocesi and {S}chacher,
\newblock {\it J. Algebra} {\bf 508} (2018) 339--363.


\bibitem{PosCones}
V. Astier and T. Unger,
\newblock Positive cones on algebras with involution,
\newblock {\it Adv. Math.} {\bf 361} (2020) 106954, 48.

\bibitem{A-U-pc-gauge}
V. Astier and T. Unger,
\newblock Positive cones and gauges on algebras with involution,
\newblock {\it Int. Math. Res. Not. IMRN} {\bf 10} (2022) 7259--7303.

\bibitem{A-U-sign}
V. Astier and T. Unger,
\newblock Signature maps from positive cones on algebras with involution,
\newblock {\it Canad. J. Math.} Published online:1--26 (2025). https://doi.org/10.4153/S0008414X25101806.

\bibitem{KSU}
M. Knebusch and C. Scheiderer,
\newblock {\it Real algebra---a first course}.
\newblock Universitext (Springer, Cham, 2022).
\newblock Translated from the 1989 German edition and with contributions by Thomas Unger.

\bibitem{BOI}
M.-A. Knus, A. Merkurjev, M. Rost, and J.-P. Tignol,
\newblock {\it The book of involutions}, 
{American Mathematical Society Colloquium Publications} {\bf 44}
\newblock (American Mathematical Society, Providence, RI, 1998).
\newblock With a preface in French by J. Tits.

\bibitem{PD}
A. Prestel and C.N. Delzell,
\newblock {\it Positive polynomials}.
\newblock Springer Monographs in Mathematics (Springer-Verlag, Berlin, 2001).
\newblock From Hilbert's 17th problem to real algebra.

\bibitem{S24}
C. Scheiderer,
\newblock {\it A course in real algebraic geometry. {Positivity} and sums of squares}, {Graduate Texts in Math.} {\bf 303} 
\newblock (Cham: Springer, 2024).


\end{thebibliography}
\end{document}